\documentclass[10pt]{amsart}
\usepackage{amsmath,amssymb,latexsym,soul,cite,amsthm,color,enumitem,graphicx,tikz, mathtools,microtype}
\usepackage[colorlinks=true,urlcolor=burgundy,citecolor=burgundy,linkcolor=burgundy,linktocpage,pdfpagelabels,bookmarksnumbered,bookmarksopen]{hyperref}
\definecolor{burgundy}{rgb}{0.5, 0.0, 0.13}\usepackage[english]{babel}
\usepackage[hyperpageref]{backref}
\usepackage[left=3.2cm,right=3.2cm,top=2.9cm,bottom=2.9cm]{geometry}

\numberwithin{equation}{section}

\newtheorem{thm}{Theorem}[section]
\theoremstyle{plain}
\newtheorem{lem}[thm]{Lemma}
\theoremstyle{plain}
\newtheorem{prop}[thm]{Proposition}
\theoremstyle{plain}

\newtheorem{definition}[thm]{Definition}
\theoremstyle{definition}
\newtheorem{rem}[thm]{Remark}
\newtheorem{ex}[thm]{Example}

\newcommand{\N}{{\mathbb N}}

\newcommand{\R}{{\mathbb R}}
\newcommand{\eps}{\varepsilon}
\newcommand{\beq}{\begin{equation}}
\newcommand{\eeq}{\end{equation}}

\renewcommand{\le}{\leqslant}
\renewcommand{\ge}{\geqslant}

\newenvironment{enumroman}{\begin{enumerate}

}{\end{enumerate}}

\DeclareMathOperator*{\esssup}{ess\, sup}
\DeclareMathOperator*{\essinf}{ess\, inf}

\title[Variational methods for partial differential inclusions]{Some reflections on variational methods for\\ partial differential inclusions}

\author[A.\ Iannizzotto]{Antonio Iannizzotto}

\address{Department of Mathematics and Computer Science
\newline\indent
University of Cagliari
\newline\indent
Viale L. Merello 92, 09123 Cagliari, Italy}
\email{antonio.iannizzotto@unica.it}

\subjclass[2010]{49J24; 49J52; 54C60.}
\keywords{Partial differential inclusions; variational methods; non-smooth analysis.}

\begin{document}

\begin{abstract}
We discuss the application of variational methods, based on non-smooth critical point theory, to a general class of partial differential inclusions.
\end{abstract}

\maketitle

\begin{center}
Version of \today\
\end{center}

\section{Introduction}\label{sec1}

\noindent
By a differential inclusion we usually mean a differential problem in which the reaction term (depending on the unknown function $u$ but not on its derivatives) is replaced by a set-valued term $F(u)$. Such problems arise quite naturally when phenomena with a high degree of uncertainty are considered, e.g., in control theory, friction dynamics, or differential game theory. The theory of differential inclusions has known an explosive development in the eighties and is now well-established (for a comprehensive account on the subject, we refer the reader to \cite{AC}, \cite{C}, \cite{HP}).
\vskip2pt
\noindent
Existence results for both ordinary and partial differential inclusions may rely on several methods, such as fixed point theory, Leray-Schauder theory, monotone operators, or approximation schemes. Such results usually require that the set-valued mapping $F$ is either upper semi-continuous or lower semi-continuous, see for instance \cite{BP}, \cite{DP}, \cite{F} (for first-order problems) and \cite{AB}, \cite{CCX}, \cite{EK}, \cite{FG}, \cite{K}, \cite{MM} (for second-order problems). In particular, upper semi-continuity fits with the significant case when $F(u)=[f_-(u),f_+(u)]$ is, pointwisely, the interval between the lower and the upper limit of some discontinuous single-valued mapping $f$, as was first noticed in \cite{F}.
\vskip2pt
\noindent
For second-order inclusions, in particular, variational methods have proved to be a powerful technique especially suitable to producing existence and multiplicity results, requiring no monotonicity assumption. In \cite{C1} a variational framework for elliptic partial second-order inclusions involving set-valued mappings of the type $F(u)=[f_-(u),f_+(u)]$ described above was established, based on the non-smooth critical point theory for locally Lipschitz continuous functionals developed in \cite{C2} (see also \cite{GP}). Then, the variational approach was extended in \cite{F1}, \cite{KRT} to the case of general set-valued reactions. We recall here some recent contribution in this direction, given in \cite{BIM}, \cite{CIM}, \cite{I2}, \cite{I}, \cite{I1}, \cite{IMM}, \cite{IP}, \cite{JS}, \cite{PP}. Most results of this type focus on the case when $F(u)=\partial J(u)$ is defined as the generalized gradient of some locally Lipschitz continuous potential $J$ -- a restriction that we shall avoid here.
\vskip2pt
\noindent
In this note we present a variational formulation for the following (partial or ordinary) differential inclusion of elliptic type with homogeneous Dirichlet boundary conditions:
\beq\label{pdi}
\begin{cases}
-\Delta u\in F(u) & \text{in $\Omega$} \\
u=0 & \text{on $\partial\Omega$.}
\end{cases}
\eeq
Here $\Omega\subset\R^N$ ($N\ge 1$) is a bounded domain with a $C^2$ boundary $\partial\Omega$, and $F:\R\to 2^\R$ is an upper semi-continuous set-valued mapping with compact, convex values, satisfying the following growth condition for convenient $a>0$, $p\in(1,2^*)$:
\beq\label{gc}
|\xi|\le a(1+|s|^{p-1}) \ \text{for all $s\in\R$, $\xi\in F(s)$}
\eeq
(we recall that $2^*=\frac{2N}{N-2}$ if $N>2$, $2^*=+\infty$ if $N\le 2$). Clearly, if $N=1$, then the Laplacian just reads as $u''$. Our treatment of problem \eqref{pdi} aims at including and, to some extent, simplifying the previous ones by using the smallest possible amount of set-valued and non-smooth analysis.
\vskip2pt
\noindent
The structure of the paper is the following: in Section \ref{sec2} we recall some well-known facts of set-valued and non-smooth analysis; in Section \ref{sec3} we introduce an energy functional and a variational framework for our problem; in Section \ref{sec4} we discuss some technical questions related to our method, including an optimality problem; and in Section \ref{sec5} we present an example dealing with a special case of problem \eqref{pdi}.

\begin{rem}
The variational approach that we are going to describe also fits with slightly different versions of problem \eqref{pdi}, e.g., with the $p$-Laplacian operator on the left-hand side, or Neumann boundary conditions, and even with non-autonomous reactions like $F(x,u)$, with the necessary adaptations.
\end{rem}

\section{Some recalls of set-valued and non-smooth analysis}\label{sec2}

\noindent
We recall some basic notions from set-valued analysis (for details see \cite{AF}, \cite{HP}). Let $X$, $Y$ be topological spaces, $F:X\to 2^Y$ be a set-valued mapping. $F$ is {\em upper semicontinuous (u.s.c.)} if, for any open set $A\subseteq Y$, the set
\[F^+(A)=\{x\in X:\,F(x)\subseteq A\}\]
is open in $X$. In the case $X=Y=\R$, upper semi-continuity of a certain class of set-valued mappings can be characterized as follows:

\begin{lem}\label{usc-cc}
If $F:\R\to 2^\R$ is a set-valued mapping with compact convex values, then the following are equivalent:
\begin{enumroman}
\item\label{usc-cc1} $F$ is u.s.c.;
\item\label{usc-cc2} $\min F,\,\max F:\R\to\R$ are l.s.c., u.s.c.\ respectively as single-valued mappings.
\end{enumroman}
\end{lem}
\begin{proof}
We first prove that \ref{usc-cc1} implies \ref{usc-cc2}. Fix $M\in\R$. The super-level set
\[\{s\in\R:\,\min F(s)>M\}=F^+(M,\infty)\]
is open, hence $\min F$ is l.s.c. In a similar way we prove that $\max F$ is u.s.c.
\vskip2pt
\noindent
Now we prove that \ref{usc-cc2} implies \ref{usc-cc1}. Let $I\subset\R$ be a bounded open interval. Then, the set
\[F^+(I)=\{s\in\R:\,\max F(s)<\sup I\}\cap\{s\in\R:\,\min F(s)>\inf I\}\]
is open. Now, let $A\subseteq\R$ be an open set. We denote by $\mathcal{I}$ the family of bounded open intervals $I\subseteq A$, hence clearly
\[\bigcup_{I\in\mathcal{I}}I=A.\]
For all $s\in F^+(A)$, the set $F(s)\subset A$ is convex and compact, hence there is $I\in\mathcal{I}$ s.t.\ $F(s)\subset I$. So the set
\[F^+(A)=\bigcup_{I\in\mathcal{I}}F^+(I)\]
is open, and $F$ turns out to be u.s.c.
\end{proof}

\noindent
A single-valued mapping $f:X\to Y$ is a {\em selection} of $F$ if $f(x)\in F(x)$ for all $x\in X$. As a special case of \cite[Theorem 8.1.3]{AF} we have the following selection theorem:

\begin{thm}\label{sel}
Let $F:\R\to 2^{\R}$ be u.s.c. Then there exists a Borel measurable selection $f:\R\to\R$ of $F$.
\end{thm}

\noindent
If $F:\R\to 2^\R$, the Aumann-type integral is a defined by
\beq\label{aumann}
\int_0^s F(\sigma)\,d\sigma=\Big\{\int_0^s f(\sigma)\,d\sigma:\,f:\R\to\R \ \text{measurable selection of $F$}\Big\}
\eeq
(note that, if $F$ is u.s.c.\ with compact convex values, then the integral is well defined, as both $\min F$ and $\max F$ are Borel measurable by Lemma \ref{usc-cc}, and it has closed convex values by \cite[Theorems 8.6.3, 8.6.4]{AF}).
\vskip2pt
\noindent
Now we recall some notions of nonsmooth critical point theory (for details see \cite{C2}, \cite{GP}). Let $(X,\|\cdot\|)$ be a Banach space, $(X^*,\|\cdot\|_*)$ be its topological dual, and $\varphi:X\to\R$ be a functional. $\varphi$ is said to be {\em locally Lipschitz continuous} if for every $u\in X$ there exist a neighborhood $U$ of $u$ and $L>0$ such that 
\[|\varphi(v)-\varphi(w)|\leq L\|v-w\| \ \mbox{for all $v,w\in U$.}\]
The {\em generalized directional derivative} of $\varphi$ at $u$ along $v\in X$ is
\[\varphi^\circ(u;v)=\limsup_{w\to u\above 0pt t\to 0^+}\frac{\varphi(w+tv)-\varphi(w)}{t}.\]
The {\em generalized gradient} (or sub-differential) of $\varphi$ at $u$ is the set
\[\partial\varphi(u)=\left\{u^*\in X^*:\, \langle u^*,v\rangle\leq\varphi^\circ(u;v) \ \mbox{for all $v\in X$}\right\}.\]
The following lemma presents some basic properties of the generalized gradient:

\begin{lem}\label{gg}
If $\varphi,\psi:X\to\R$ are locally Lipschitz continuous, then
\begin{enumroman}
\item\label{gg1} $\partial\varphi(u)$ is convex, closed and weakly$^*$ compact for all $u\in X$;
\item\label{gg2} $\partial\varphi:X\to 2^{X^*}$ is an upper semicontinuous set-valued mapping with respect to the weak$^*$ topology on $X^*$;
\item\label{gg3} if $\varphi\in C^1(X)$, then $\partial\varphi(u)=\{\varphi'(u)\}$ for all $u\in X$;
\item\label{gg4} $\partial(\lambda\varphi)(u)=\lambda\partial\varphi(u)$ for all $\lambda\in\R$, $u\in X$;
\item\label{gg5} $\partial(\varphi+\psi)(u)\subseteq\partial\varphi(u)+\partial\psi(u)$ for all $u\in X$;
\item\label{gg6} for all $u,v\in X$ there exists $u^*\in\partial\varphi(u)$ such that  $u^*(v)=\varphi^\circ(u;v)$;
\item\label{gg7} if $u$ is a local minimizer (or maximizer) of $\varphi$, then $0\in\partial\varphi(u)$.
\end{enumroman}
\end{lem}

\noindent
By Lemma \ref{gg} \ref{gg1}, we may define for all $u\in X$
\begin{equation}\label{m}
m_\varphi(u)=\min_{u^*\in\partial\varphi(u)}\|u^*\|_*,
\end{equation}
We say that $u\in X$ is a {\em (generalized) critical point} of $\varphi$ if $m_\varphi(u)=0$ (i.e. $0\in\partial\varphi(u)$). The set of all critical points of $\varphi$ is denoted by $K(\varphi)$, while $K_c(\varphi)$ is the set of all critical points $u$ s.t.\ $\varphi(u)=c$ ($c\in\R$). We recall the non-smooth {\em Palais-Smale condition}, where $(u_n)$ denotes a sequence in $X$:
\beq
\tag{PS}\label{ps} \text{$(\varphi(u_n))$ bounded, $m_\varphi(u_n)\to 0$} \ \Longrightarrow \ \text{$(u_n)$ has a convergent subsequence.}
\eeq
Finally, we recall the non-smooth version of the mountain pass theorem (see \cite[Theorem 2.1.3]{GP}):

\begin{thm}\label{mpt}
Let $X$ be a Banach space, $\varphi:X\to\R$ be a locally Lipschitz functional satisfying \eqref{ps}, $\bar u\in X$, $0<r<\|\bar u\|$ be s.t.
\[\max\{\varphi(0),\varphi(\bar u)\}\le \inf_{\|u\|=r}\varphi(u),\]
and let
\[\Gamma=\{\gamma\in C([0,1],X):\,\gamma(0)=0,\,\gamma(1)=\bar u\},\]
\[c=\inf_{\gamma\in\Gamma}\,\max_{t\in[0,1]}\varphi(\gamma(t)).\]
Then, $c\ge\inf_{\|u\|=r}\varphi(u)$ and $K_c(\varphi)\neq\emptyset$.
\end{thm}

\section{A variational framework: the 'shrinking' method}\label{sec3}

\noindent
In this section we introduce a variational formulation for problem \eqref{pdi}. Our approach follows that of \cite{KRT}, the main difference being that we do not prescribe the choice of a special selection. We will make use of the Sobolev space $H^1_0(\Omega)$, endowed with the norm $\|u\|=\|\nabla u\|_2$ (by $\|\cdot\|_\nu$ we denote the usual norm of $L^\nu(\Omega)$, for any $\nu\in[1,\infty]$).
\vskip2pt
\noindent
We begin by defining a notion of solution:

\begin{definition}\label{sol}
We say that $u\in H^1_0(\Omega)$ is a (weak) solution of \eqref{pdi}, if there exists $w\in L^\nu(a,b)$ ($\nu>1$) s.t.
\begin{enumroman}
\item\label{sol1} $\displaystyle\int_\Omega\nabla u\cdot\nabla v\,dx=\int_\Omega wv\,dx$ for all $v\in H^1_0(\Omega)$;
\item\label{sol2} $w(x)\in F(u(x))$ for a.e. $x\in\Omega$.
\end{enumroman}
\end{definition}

\noindent
Definition \ref{sol} above is quite natural, although it does not lead to a regularity theory, as in general the function $w\in L^\nu(\Omega)$ remains undetermined.
\vskip2pt
\noindent
We will now reduce problem \eqref{pdi} to a typically variational one (that was studied in \cite{C1}), by 'shrinking' pointwise the set $F(u)$ to a smaller interval, which happens to be the gradient of a convenient locally Lipschitz continuous potential. We will seek solutions to this reduced inclusion problem, which turn out to solve \eqref{pdi} as well.
\vskip2pt
\noindent
By Theorem \ref{sel} there exists a Borel measurable mapping $f:\R\to\R$ s.t.\ $f(s)\in F(s)$ for all $s\in\R$. By \eqref{gc}, we clearly have $f\in L^\infty_{\rm loc}(\R)$, so we can define a locally Lipschitz continuous potential $J_f:\R\to\R$ by setting for all $s\in\R$
\[J_f(s)=\int_0^s f(\sigma)\,d\sigma.\]
From \cite[Example 2.2.5]{C2} we see that the gradient of $J_f$ at any $s\in\R$ is given by
\beq\label{grad}
\partial J_f(s)=[f_-(s),f_+(s)],
\eeq
where we have set
\[f_-(t)=\lim_{\delta\to 0^+}\essinf_{|t-s|<\delta}f(t), \ f_+(t)=\lim_{\delta\to 0^+}\esssup_{|t-s|<\delta}f(t).\]
We can now define an energy functional for problem \eqref{pdi} by setting for all $u\in H^1_0(\Omega)$
\beq\label{phi}
\varphi(u)=\frac{\|\nabla u\|_2^2}{2}-\int_\Omega J_f(u)\,dx.
\eeq

\begin{prop}\label{vf}
The functional $\varphi$ defined in \eqref{phi} is locally Lipschitz continuous in $H^1_0(\Omega)$. Moreover, if $u\in K(\varphi)$, then $u$ is a solution of problem \eqref{pdi}.
\end{prop}
\begin{proof}
Clearly, since $H^1_0(\Omega)$ is a Hilbert space, the functional $u\mapsto\frac{\|u\|^2}{2}$ is of class $C^1$ (in particular locally Lipschitz continuous) with derivative $A:H^1_0(\Omega)\to H^{-1}(\Omega)$ given for all $u,v\in H^1_0(\Omega)$ by
\[A(u)(v)=\int_\Omega \nabla u\cdot\nabla v\,dx.\]
Set for all $u\in L^p(\Omega)$
\[\psi(u)=\int_\Omega J_f(u)\,dx.\]
From \eqref{gc} we have for all $s\in\R$
\[|f(s)|\le a(1+|s|^{p-1}),\]
which implies
\[|J_f(s)|\le C(1+|s|^p)\]
($C>0$ will denote a constant, whose value may change from line to line). So, $\psi:L^p(\Omega)\to\R$ is well defined. Let us prove that $\psi$ is Lipschitz continuous in any bounded subset of $L^p(\Omega)$. Fix $M>0$, $u,v\in L^p(\Omega)$ s.t. $\|u\|_p,\|v\|_p\le M$, then
\begin{align*}
|\psi(u)-\psi(v)| &\le \int_\Omega\Big|\int_v^u f(s)\,ds\Big|\,dx \\
&\le a\int_\Omega(1+|u|^{p-1}+|v|^{p-1})|u-v|\,dx \\
&\le C(1+\|u\|_p^{p-1}+\|v\|_p^{p-1})\|u-v\|_p \\
&\le C\|u-v\|_p.
\end{align*}
Let $u\in L^p(\Omega)$, $w^*\in\partial\psi(u)$. By \cite[Theorem 4.11]{B} there exists $w\in L^{p'}(\Omega)$ s.t.\ for all $v\in L^p(\Omega)$
\[w^*(v)=\int_\Omega wv\,dx.\]
Moreover, by \cite[Theorem 1.3.10]{GP} we have for a.e. $x\in\Omega$
\[w(x)\in J_f(u(x)).\]
An expressive way to rephrase the above passages is the following formula:
\beq\label{ac}
\partial\psi(u)\subseteq\int_\Omega\partial J_f(u)\,dx,
\eeq
where the integral is defined in the sense of \eqref{aumann}. Since the embedding $H^1_0(\Omega)\hookrightarrow L^p(\Omega)$ is continuous and dense, by \cite[Proposition 1.3.17]{GP} the functional $\psi$ is locally Lipschitz continuous in $H^1_0(\Omega)$ as well, with the same gradient.
\vskip2pt
\noindent
By Lemma \ref{gg} \ref{gg3} - \ref{gg5}, the functional $\varphi:H^1_0(\Omega)\to\R$ is locally Lipschitz continuous and by \eqref{ac} its gradient satisfies at any $u\in H^1_0(\Omega)$ by
\[\partial\varphi(u)\subseteq A(u)-\int_\Omega\partial J_f(u)\,dx.\]
Now, assume that $u\in K(\varphi)$. Then $0\in\partial\varphi(u)$, i.e.,
\[A(u)\in\int_\Omega\partial J_f(u)\,dx.\]
Explicitly, we can find $w\in L^{p'}(\Omega)$ s.t.\ for all $v\in H^1_0(\Omega)$
\[A(u)(v)=\int_\Omega wv\,dx,\]
and $w(x)\in\partial J_f(u(x))$ for a.e. $x\in\Omega$. By virtue of \eqref{grad}, this implies
\beq\label{in}
f_-(u(x))\le w(x)\le f_+(u(x)).
\eeq
Recalling the definition of $f_\pm$ and Lemma \ref{usc-cc}, we have for all $s\in\R$
\[\min F(s)\le f_-(s)\le f_+(s)\le \max F(s),\]
and by convexity we also have
\[F(s)=[\min F(s),\,\max F(s)],\]
so \eqref{in} implies $w(x)\in F(u(x))$ for a.e. $x\in\Omega$. Thus, both conditions of Definition \ref{sol} are satisfied, and $u$ is a solution of \eqref{pdi}.
\end{proof}

\noindent
Moreover, $\varphi$ satisfies a \eqref{ps}-type pre-compactness property:

\begin{lem}\label{pps}
Let $(u_n)$ be a bounded sequence in $H^1_0(\Omega)$ s.t.\ $(\varphi(u_n))$ is bounded and $m_\varphi(u_n)\to 0$. Then, $(u_n)$ has a convergent subsequence.
\end{lem}
\begin{proof}
By \eqref{m}, for all $n\in\N$ there exists $u_n\in\partial\varphi(u_n)$ s.t. $\|u_n\|_*=m_\varphi(u_n)$. Reasoning as in Proposition \ref{vf}, we have
\[u^*_n=A(u_n)-w_n\]
for some $w_n\in L^{p'}(\Omega)$ satisfying $w_n(x)\in\partial F(u_n(x))$ for a.e. $x\in\Omega$. Due to reflexivity of $H^1_0(\Omega)$ and the compact embedding $H^1_0(\Omega)\hookrightarrow L^p(\Omega)$, passing if necessary to a subsequence we have $u_n\rightharpoonup u$ in $H^1_0(\Omega)$ and $u_n\to u$ in $L^p(\Omega)$ for some $u\in H^1_0(\Omega)$. Besides, by \eqref{gc} we see that $(w_n)$ is bounded in $L^{p'}(\Omega)$.
\vskip2pt
\noindent
By what stated above, we have for all $n\in\N$
\begin{align*}
\|u_n-u\|^2 &= A(u_n)(u_n-u)-A(u)(u_n-u) \\
&= u^*_n(u_n-u)+\int_\Omega w_n(u_n-u)\,dx-A(u)(u_n-u) \\
&\le m_\varphi(u_n)\|u_n-u\|+\|w_n\|_{p'}\|u_n-u\|_p-A(u)(u_n-u),
\end{align*}
and the latter tends to 0 as $n\to\infty$. Thus, $u_n\to u$ in $H^1_0(\Omega)$.
\end{proof}

\section{Some technical remarks}\label{sec4}

\noindent
As seen in Section \ref{sec3}, the key step in order to deal with problem \eqref{pdi} variationally is to 'shrink' the set-valued term $F(u)$ to a generalized gradient, which admits the twofold representation $\partial J_f(u)=[f_-(u),f_+(u)]$. We denote $\mathcal{S}$ the family of all u.s.c.\ set-valued mappings from $\R$ into itself, with compact convex values.
\vskip2pt
\noindent
The first question we shall address in the present section, is whether such 'shrinking' is non-trivial. Indeed, there are mappings in $\mathcal{S}$ which are {\em not} gradients, ad the following example shows.

\begin{ex}
Let $F:\R\to 2^\R$ be defined by putting for all $s\in\R$
\[F(s)=\begin{cases}
\{-1\} & \text{if $s<0$} \\
[-2,2] & \text{if $s=0$} \\
\{1\} & \text{if $s>0$.}
\end{cases}\]
Such mapping is u.s.c.\ with compact convex values, and clearly satisfies \eqref{gc} with any $p>1$. We show that $F$ cannot be a gradient in itself, arguing by contradiction: let $J:\R\to\R$ be locally Lipschitz continuous s.t. $\partial J(s)=F(s)$ for all $s\in\R$. Then, in particular, $J$ has a single-valued gradient at any $s\neq 0$, and by \cite[Theorem 2.5.1]{C2} we should have
\[\partial J(0)=\,\text{convex hull of}\,\{-1,1\}=[-1,1],\]
against $F(0)=[-2,2]$. Then, let us proceed as in Section \ref{sec3} by choosing a Borel-measurable selection $f:\R\to\R$ of $F$, which must be of the following type:
\[f(s)=\begin{cases}
-1 & \text{if $s<0$} \\
\text{any value in $[-2,2]$} & \text{if $s=0$} \\
1 & \text{if $s>0$.}
\end{cases}\]
Its potential (independent of the value at 0) is then $J_f:\R\to\R$ defined for all $s\in\R$ by
\[J_f(s)=|s|,\]
whose gradient coincides wit $F(s)$ at any $s\neq 0$ while $\partial J_f(0)=[-1,1]$.
\end{ex}

\noindent
The second question about the 'shrinking' approach, is whether it can be optimal. To be precise, let us define a partial ordering in the set $\mathcal{S}$, by setting $F\lesssim G$ for any $F,G\in\mathcal{S}$ s.t. $F(s)\subseteq G(s)$ for all $s\in\R$. We know from Section \ref{sec3} that for any $F\in\mathcal{S}$ there exists a locally Lipschitz continuous $J$ s.t. $\partial J\lesssim F$. So, the question is whether gradients of Lipschitz potentials are {\em minimal} elements in $\mathcal{S}$ with respect to the ordering defined above.
\vskip2pt
\noindent
Certain classes of gradients exhibit some extremality property: for instance, gradients of continuous convex potentials are maximal monotone operators (see \cite[Theorem 1.4.1]{GP}); and if $J$ is strictly differentiable at any point, then $\partial J$ is single-valued (see \cite[Proposition 2.2.4]{C2}).
\vskip2pt
\noindent
Nevertheless, the answer is in general negative, as the following simple example shows:

\begin{ex}
Let $F:\R\to 2^\R$ be defined by $F(s)=[0,1]$ for all $s\in\R$. From \cite{R} we know that there exists a Borel-measurable set $A\subset\R$, s.t.\ for all bounded interval $I\subset\R$ both $A\cap I$ and $A^c\cap I$ have positive measure. Let us set
\[f_1(s)=\begin{cases}
\frac{1}{3} & \text{if $s\in A$} \\
\frac{2}{3} & \text{if $s\in A^c$}
\end{cases}, \ 
f_2(s)=\begin{cases}
\frac{1}{4} & \text{if $s\in A$} \\
\frac{3}{4} & \text{if $s\in A^c$}
\end{cases}.\]
Both are Borel-measurable selections of $F$ and they induce potentials, which have set-valued integrals given at any $s\in\R$ by
\[\partial J_{f_1}(s)=\Big[\frac{1}{3},\,\frac{2}{3}\Big], \ \partial J_{f_2}(s)=\Big[\frac{1}{4},\,\frac{3}{4}\Big],\]
respectively. Thus, $\partial J_{f_1}(s)\subset \partial J_{f_2}(s)$ (strictly) at all $s\in\R$. In particular, $\partial J_{f_2}$ is not minimal with respect to the ordering $\lesssim$.
\end{ex}

\noindent
We conclude this section by presenting some examples of Borel-measurable selections which can be chosen from a general u.s.c.\ set-valued mapping $F:\R\to 2^\R$ with compact convex values:
\[f_1(s)=\min F(s), \ f_2(s)=\max F(s), \ f_3(s)=\frac{1}{2}(\min F(s)+\max F(s)),\]
\[f_4(s)=\begin{cases}
\max F(s) & \text{if $s<0$} \\
\min F(s) & \text{if $s\ge 0$.}
\end{cases}\]
The mappings $f_1$, $f_2$ above represent the so-called {\em extremal selections}, which are in general appreciated in the theory of diffierential inclusions, see \cite{AF}, \cite{T}. The mapping $f_4$ was introduced in \cite{KRT} and it has the following special property related to the set-valued integral defined in \eqref{aumann}: for all $s\in\R$
\[\int_0^s f_4(\sigma)\,d\sigma=\min\int_0^s F(\sigma)\,d\sigma.\]
Finally, we would like to remark that, though we introduce a single-valued mapping $f:\R\to\R$, critical points of the resulting energy functional {\em do not}, in any sense, solve the elliptic equation
\beq\label{pde}
-\Delta u=f(u) \ \text{in $\Omega$,}
\eeq
due to the discontinuity of $f$. Indeed, nothing prevents $u$ from taking in a subset of $\Omega$ with positive measure values at which $f_-<f_+$. In order to find solutions of \eqref{pde} we need further sign assumptions which allow us to avoid the 'jumps' of $f$ (see for instance \cite{BIM}, \cite{MM1}).

\section{A case study}\label{sec5}

\noindent
In this final section we present the variational study of a problem of the type \eqref{pdi}. Namely, we will deal with the following Dirichlet problem for an elliptic differential inclusion depending on a parameter $\lambda>0$:
\beq\label{pl}
\begin{cases}
-\Delta u\in\lambda F(u) & \text{in $\Omega$} \\
u=0 & \text{on $\partial\Omega$.}
\end{cases}
\eeq
Here $\Omega$ is as in the Introduction, while $F:\R\to 2^\R$ is u.s.c.\ with compact convex values, satisfying \eqref{gc} and the following additional hypotheses:
\beq\label{zero}
\lim_{s\to 0}\max_{\xi\in F(s)}\frac{\xi}{s}=0,
\eeq
\beq\label{inf}
\limsup_{|s|\to +\infty}\min_{\xi\in F(s)}\frac{\xi}{s}=0.
\eeq
and there exists $\bar s>0$ s.t.
\beq\label{ss}
\max F(s)>0 \ \text{for all $s\in(0,\bar s]$.}
\eeq
Hypotheses \eqref{zero} and \eqref{inf} can be rephrased by saying that a convenient selection of $F$ is super-linear at 0 and sub-linear at $\infty$ (in particular, \eqref{zero} implies that \eqref{pl} always admits the zero solution), while \eqref{ss} ensures the existence of negative values for the energy associated to problem \eqref{pl}.
\vskip2pt
\noindent
Our result is an easy one, but it provides an example of how variational methods can be applied to general differential inclusions, especially when dealing with multiple solutions (for more advanced results, employing several tools in non-smooth critical point theory and non-smooth Morse theory, see \cite{BIM}, \cite{CIM}, \cite{I2}, \cite{IMM}, \cite{IP}, and \cite{KRT} where a set-valued version of the celebrated \cite[Theorem 2.15]{R1} is presented).

\begin{thm}\label{app}
Let $F$ satisfy \eqref{gc}, \eqref{zero}, \eqref{inf}, and \eqref{ss}. Then, there exists $\lambda^*>0$ s.t.\ for all $\lambda\ge\lambda^*$ problem \eqref{pl} admits at least two non-zero solutions.
\end{thm}
\begin{proof}
First we define a Borel-measurable selection $f:\R\to\R$ of $F$ by setting for all $s\ge 0$
\[f(s)=\begin{cases}
\max F(s) & \text{if $s\le -\bar s$} \\
\min F(s) & \text{if $s\in (-\bar s,0]$} \\
\max F(s) & \text{if $s\in(0,\bar s]$} \\
\min F(s) & \text{if $s>\bar s$}
\end{cases}\]
We define the locally Lipschitz continuous potential $J_f:\R\to\R$ as in Section \ref{sec3}, and for all $\lambda>0$ and $u\in H^1_0(\Omega)$ we set
\[\varphi_\lambda(u)=\frac{\|\nabla u\|_2^2}{2}-\lambda\int_\Omega J_f(u)\,dx.\]
From hypotheses \eqref{gc} and \eqref{zero} we easily deduce that, for all $\eps>0$, there exists $C_\eps>0$ s.t.\ for all $s\in\R$
\beq\label{eps}
|J_f(s)|\le\eps|s|^2+C_\eps|s|^p.
\eeq
We recall that, for all $\nu\in[1,2^*[$, the embedding $H^1_0(\Omega)\hookrightarrow L^\nu(\Omega)$ is continuous and compact, and we denote $c_\nu>0$ the best constant s.t.\ $\|u\|_\nu\le c_\nu\|u\|$ for all $u\in H^1_0(\Omega)$. Now fix $\lambda>0$. For all $r>0$ and $u\in H^1_0(\Omega)$ s.t. $\|u\|=r$, by \eqref{eps} we have
\begin{align*}
\varphi_\lambda(u) &\ge \frac{\|u\|^2}{2}-\lambda\int_\Omega(\eps u^2+C_\eps |u|^p)\,dx \\
&\ge \Big(\frac{1}{2}-\lambda\eps c_2^2\Big)r^2-\lambda C_\eps c_p^p r^p.
\end{align*}
Take $\eps<\frac{1}{2\lambda c_2^2}$. It is easily seen that the mapping $h(r):=(\frac{1}{2}-\lambda \eps c_2^2)r^2-\lambda C_\eps c_p^p r^p$ has a strict local minimum at 0, hence we have $h(r)>0$ for any $r>0$ small enough. Thus,
\beq\label{ring}
\inf_{\|u\|=r}\varphi_\lambda(u)=:\eta_r>0.
\eeq
Now we prove that $\varphi_\lambda$ is coercive in $H^1_0(\Omega)$. Indeed, by \eqref{inf}, for all $\delta>0$ we can find $M>0$ s.t. $J_f(s)\le\delta s^2$ for all $|s|>M$. Then, for all $u\in H^1_0(\Omega)$ we have
\begin{align*}
\varphi_\lambda(u) &\ge \frac{\|\nabla u\|_2^2}{2}-\lambda\int_{\{|u|\le M\}}J_f(u)\,dx-\lambda\int_{\{|u|> M\}}\delta u^2\,dx \\
&\ge \Big(\frac{1}{2}-\lambda\delta c_2^2\Big)\|u\|^2-C.
\end{align*}
Taking $\delta<\frac{1}{2\lambda c_2^2}$, we see that
\beq\label{coer}
\lim_{\|u\|\to\infty}\varphi_\lambda(u)=+\infty.
\eeq
The functional $\varphi_\lambda$ is sequentially weakly l.s.c.\ in $H^1_0(\Omega)$ as the sum of a continuous convex functional and of an integral functional which turns out to be continuous in $L^p(\Omega)$ (recall that $H^1_0(\Omega)\hookrightarrow L^p(\Omega)$ is compact). Thus, by \eqref{coer} it admits a global minimizer, i.e., we can find $u_1\in H^1_0(\Omega)$ s.t.
\beq\label{min}
\varphi_\lambda(u_1)=\inf_{u\in H^1_0(\Omega)}\varphi_\lambda(u).
\eeq
By Lemma \ref{gg} \ref{gg7} we have $u_1\in K(\varphi_\lambda)$. Now we prove that, for $\lambda>0$ big enough, $\varphi_\lambda$ takes negative values. Let $\bar u\in H^1_0(\Omega)$ be s.t. $\bar u(x)=\bar s$ in some closed ball $\bar B_\rho(\bar x)$, $\bar u(x)=0$ in $\Omega\setminus B_{2\rho}(\bar x)$, and $0\le\bar u(x)\le\bar s$ in $\Omega$ (with $\bar x\in\Omega$, $\rho>0$ s.t. $\bar B_{2\rho}(\bar x)\subset\Omega$. Then, by \eqref{ss} we have
\begin{align*}
\varphi_\lambda(\bar u) &= \frac{\|\nabla\bar u\|_2^2}{2}-\lambda\int_{B_\rho(\bar x)}J_f(\bar s)\,dx-\lambda\int_{B_{2\rho}(\bar x)\setminus\bar B_\rho(\bar x)}J_f(\bar u)\,dx \\
&\le \frac{\|\nabla\bar u\|_2^2}{2}-\lambda J_f(\bar s)|B_\rho(\bar x)|,
\end{align*}
and the latter is negative whenever $\lambda>\lambda^*$, with
\[\lambda^*=\frac{\|\nabla\bar u\|_2^2}{2J_f(\bar s)|B_\rho(\bar x)|}.\]
So we have $\varphi_\lambda(\bar u)<0$. By \eqref{min}, this implies $u_1\neq 0$, hence by Lemma \ref{vf} $u_1$ is a non-zero solution of \eqref{pl}.
\vskip2pt
\noindent
Moreover, taking $r<\|\bar u\|$, by \eqref{ring} we have
\[\max\{\varphi_\lambda(0),\,\varphi_\lambda(\bar u)\}<\eta_r,\]
hence $\varphi_\lambda$ realizes the mountain pass geometry. Let us now check that $\varphi_\lambda$ satisfies \eqref{ps}. Let $(u_n)$ be a sequence in $H^1_0(\Omega)$, s.t. $(\varphi_\lambda(u_n))$ is bounded and $m_{\varphi_\lambda}(u_n)\to 0$. By \eqref{coer}, $(u_n)$ is bounded in $H^1_0(\Omega)$, so by Lemma \ref{pps} it admits a convergent subsequence.
\vskip2pt
\noindent
From Theorem \ref{mpt} we can deduce the existence of $u_2\in K_c(\varphi_\lambda)$ with $c\ge\eta_r$, in particular $c>0$ by \eqref{ring}. So, $u_2\neq 0,u_1$ is a second non-zero solution of \eqref{pl}, which concludes the proof.
\end{proof}

\noindent
We conclude by presenting an example of a set-valued mapping satisfying the hypotheses of Theorem \ref{app}:

\begin{ex}
Let $F:\R\to 2^\R$ be defined by setting for all $s\ge 0$
\[F(s)=\begin{cases}
[0,\ln(s^2+1)] & \text{if $s\in[0,1]$} \\
[0,1] & \text{if $s=1$} \\
[\ln(s)+1,s^2] & \text{if $s>1$,}
\end{cases}\]
and $F(s)=-F(-s)$ for all $s<0$. It is easily seen that $F$ satisfies assumptions \eqref{gc}, \eqref{zero}, \eqref{inf}, and \eqref{ss} (note that $F$ is not defined as the generalized gradient of any locally Lipschitz potential). Hence, Theorem \ref{app} ensures the existence of at least two non-zero solutions of problem \eqref{pl} for all $\lambda>0$ big enough.
\end{ex}

\end{document}